\documentclass[12pt,reqno]{amsart}
\usepackage{amscd}

\newtheorem{theorem}{Theorem}
\newtheorem{lem}[theorem]{Lemma}

\newcommand{\nexteq}{\displaybreak[0]\\ &=}
\newcommand{\nexteqv}{\displaybreak[0]\\ &\iff}

\newcommand{\qbinom}[2]{\genfrac{[}{]}{0pt}{}{#1}{#2}}

\newcommand{\cA}{\mathcal{A}}
\newcommand{\cB}{\mathcal{B}}
\newcommand{\cC}{\mathcal{C}}

\DeclareMathOperator{\PG}{PG}
\DeclareMathOperator{\GF}{GF}

\begin{document}
\title{Godsil--McKay switching and twisted Grassmann graphs}
\author{Akihiro Munemasa}
\dedicatory{Dedicated to Andries E. Brouwer on the occasion of his 
65th birthday}

\address{Research Center for Pure and Applied Mathematics\\
Graduate School of Information Sciences\\
Tohoku University}
\email{munemasa@math.is.tohoku.ac.jp}
\date{December 31, 2015}
\maketitle
\begin{abstract}
We show that the twisted Grassmann graphs introduced by
Van Dam and Koolen are obtained by Godsil--McKay switching
applied to the Grassmann graphs. The partition for the
switching is constructed by a polarity of a hyperplane.
\end{abstract}

\section{Introduction}\label{sec:1}
The twisted Grassmann graphs introduced
by Van Dam and Koolen \cite{vDK},
are the first family of non-vertex-transitive 
distance-regular graphs with unbounded diameter. 
We refer the reader to \cite{BI,BCN,DKT} for an extensive discussion of 
distance-regular graphs, to \cite{M} for a characterization of 
Grassmann graphs, and to \cite{BFK,FKT} 
for more information on the twisted Grassmann graphs.

Let $V$ be a $(2e+1)$-dimensional vector space over a finite
field $\GF(q)$.
If $W$ is a subset of $V$ closed under multiplication by the
elements of $\GF(q)$, then we denote by $[W]$ the set of
$1$-dimensional subspaces contained in $W$.
We also denote by 
$\qbinom{W}{k}$ the set of $k$-dimensional subspaces of $W$,
when $W$ is a vector space.
The Grassmann graph $J_q(2e+1,e+1)$ is the graph with vertex set
$\qbinom{V}{e+1}$, where two vertices $W_1,W_2$ 
are adjacent whenever $\dim W_1\cap W_2=e$.

Let $H$ be a fixed hyperplane of $V$.
The twisted Grassmann graph $\tilde{J}_q(2e+1,e)$ 
has $\cA\cup\cB$ as the set of vertices, where
\begin{align}
\cA&=\{W\in\qbinom{V}{e+1}\mid W\not\subset H\},
\label{A}\\
\cB&=\qbinom{H}{e-1}.
\notag
\end{align}
The adjacency is defined as follows:
\[
W_1\sim W_2\iff
\begin{cases}
\dim W_1\cap W_2=e&\text{if }W_1\in\cA,\;W_2\in\cA,\\
W_1\supset W_2&\text{if }W_1\in\cA,\;W_2\in\cB,\\
\dim W_1\cap W_2=e-2&\text{if }W_1\in\cB,\;W_2\in\cB.
\end{cases}
\]

Let $\sigma$ be a polarity of $H$. That is,
$\sigma$ is an inclusion-reversing permutation of
the set of subspaces of $H$, such that $\sigma^2$ is the identity.
The pseudo-geometric design constructed by Jungnickel and
Tonchev \cite{JT} has $[V]$ as the set of points, and $\cA'\cup\cB'$
as the set of blocks, where
\begin{align*}
\cA'&=\{[\sigma(W\cap H)]\cup[W\setminus H]\mid W\in\cA\},\\
\cB'&=\{[W]\mid W\in\qbinom{H}{e+1}\}.
\end{align*}
It is shown in \cite{JT} that the incidence structure
$([V],\cA'\cup\cB')$ is a $2$-$(v,k,\lambda)$ design, where
\[
v=\frac{q^{2e+1}-1}{q-1},\;
k=\frac{q^{e+1}-1}{q-1},\;
\lambda=\frac{(q^{2e-1}-1)\cdots(q^{e+1}-1)}{(q^{e-1}-1)\cdots(q-1)}.
\]
Recall that the geometric design $\PG_e(2e,q)$ has $[V]$
as the set of points, and $\qbinom{V}{e+1}$ as the set of blocks.
Jungnickel and Tonchev \cite{JT} describe
this design as the one obtained from the geometric design
$\PG_e(2e,q)$ by distorting with the help of a polarity acting
on a fixed hyperplane in $\PG(2e,q)$.
The sizes of the intersections of pairs of blocks are
\[\frac{q^i-1}{q-1}\quad(i=1,\dots,e),\]
in both geometric design and pseudo-geometric design.
The block graph of $\PG_e(2e,q)$, where two distinct
blocks are adjacent 
whenever their intersection has size $(q^e-1)/(q-1)$,
is nothing but the Grassmann graph $J_q(2e+1,e+1)$.
The block graph $\Delta(e,q)$, defined in a similar manner
for the pseudo-geometric design, is shown to be
isomorphic to the 
twisted Grassmann graph $\tilde{J}_q(2e+1,e)$
by the author and Tonchev \cite{MT}.

In this paper, we show that $\Delta(e,q)$ is obtained from
the Grassmann graph $J_q(2e+1,e+1)$ via Godsil--McKay switching.
The following commutative diagram illustrates the situation.
\[\begin{CD}
\PG_{e}(2e,q) @>\text{block graph}>> J_q(2e+1,{e+1})\\
@V\text{distort}VV  @V\text{{{GM switching}}}VV\\
\text{\small pseudo-geometric design} @>\text{block graph}>> 
\Delta(e,q)\cong\tilde{J}_q(2e+1,e)
\end{CD}\]
It is worth mentioning that Van Dam, Haemers, Koolen and Spence
\cite{vDHKS} constructed a large number of graphs cospectral with
distance-regular graphs including the Grassmann graphs, 
by Godsil--McKay switching and distorting
the set of lines of a partial linear space.
The contribution of the present paper is simply that
distorting which leads to the construction of the
twisted Grassmann graph can also be described by
Godsil--McKay switching.

\section{Godsil--McKay switching}

Let $\Gamma$ be a graph with vertex set $X$, and let 
$\{C_1,\dots,C_t,D\}$ be a partition of $X$ such
that $\{C_1,\dots,C_t\}$ is an equitable partition of the
subgraph induced on $X \setminus D$.
This means that the number of neighbors in $C_i$ of a vertex $x$ depends
only on $j$ for which $x\in C_j$ holds, and independent of the choice of $x$
as long as $x\in C_j$. Assume also that
for any $x \in D$ and 
$i \in \{1,\dots,t\}$,
$x$ has either $0, \frac12|C_i|$ or $|C_i|$ neighbors in $C_i$. 
The graph $\tilde{\Gamma}$ obtained by interchanging adjacency and 
nonadjacency between $x\in D$ and the vertices in $C_i$ whenever 
$x$ has $\frac12 |C_i|$ neighbors in $C_i$, is cospectral
with $\Gamma$ (see \cite{GM}). The operation of constructing $\tilde{\Gamma}$
from $\Gamma$ is called Godsil--McKay switching.

Godsil--McKay switching has been known for years, as a method
to construct cospectral graphs. However, finding an instance
for which the hypotheses of Godsil--McKay switching are satisfied,
is nontrivial.
We mention recent work \cite{ABH,AH}, but mainly the case $t=1$
has been treated. In our work, $t$ will be unbounded.

We now take $\Gamma$ to be the Grassmann graph 
$J_q(2e+1,e+1)$, 
and keep the same notation as in Section~\ref{sec:1}.
Recall that $H$ is a fixed hyperplane of $V$, and
the set $\cA$ is defined in \eqref{A}.

\begin{lem}\label{lem:1}
Let $\{C_1,C_2,\dots,C_t\}$ be an equitable partition of the graph
$J_q(2e,e)$ with vertex set $\qbinom{H}{e}$. Let
\[
\tilde{C}_i=\{W\in\qbinom{V}{e+1}\mid W\cap H\in C_i\}\quad(1\leq i\leq t).
\]
Then 
$\{\tilde{C}_1,\tilde{C}_2,\dots,\tilde{C}_t\}$ is an equitable partition
of the subgraph $J_q(2e+1,e+1)$ induced on $\cA$.
\end{lem}
\begin{proof}
By the assumption, for $1\leq i,j\leq t$,
there exists an integer $m_{ij}$ such that
\[
|\{U\in C_i\mid \dim U\cap U'=e-1\}|=m_{ij}
\]
whenever $U'\in C_j$.
Suppose $U\in C_i$ and $U'\in C_j$ satisfy
$\dim U\cap U'=e-1$. For $W'\in\tilde{C}_j$ with $W'\cap H=U'$, 
counting the number of elements in the set
\[\{(Z,W)\in [W'\setminus H]\times\qbinom{V}{e+1}
\mid W\cap H=U,\;Z\subset W\}\]
in two ways, we find
\[q^e=q^{e-1}|\{
W\in \qbinom{V}{e+1}\mid W\cap H=U,\;W\cap W'\not\subset H\}|.\]
Thus,
\begin{align*}
&|\{W\in \tilde{C}_i\mid \dim W\cap W'=e\}|
\\&=\sum_{U\in C_i}
|\{W\in \qbinom{V}{e+1}\mid W\cap H=U,\;\dim W\cap W'=e\}|
\nexteq
\sum_{\substack{U\in C_i\\ \dim U\cap U'=e-1}}
|\{W\in \qbinom{V}{e+1}\mid W\cap H=U,\;W\cap W'\not\subset H\}|
\\&\quad+\delta_{ij}|\{W\in\qbinom{V}{e+1}\mid W\cap H=U'=W\cap W'\}|
\nexteq
qm_{ij}+\delta_{ij}
|\{\bar{W}\in[V/U']\mid \bar{W}\not\subset[H/U']\cup[W'/U']\}|
\nexteq
qm_{ij}+\delta_{ij}
(|[V/U']|-|[H/U']|-1)
\nexteq
qm_{ij}+\delta_{ij}(q^e-1).
\end{align*}
Therefore, every vertex in $\tilde{C}_i$ has exactly 
$qm_{ij}+\delta_{ij}(q^e-1)$
neighbors in $\tilde{C}_j$.
\end{proof}

Let
\begin{align*}
C_U&=\{W\in\cA\mid W\cap H=U\}\quad(U\in\qbinom{H}{e}),
\\
D&=\qbinom{H}{e+1},\\
\cC&=\{C_U\cup C_{\sigma(U)}\mid U\in\qbinom{H}{e}\},
\end{align*}
where $\sigma$ is a polarity of $H$.
Then
\[
\cA=\bigcup_{U\in\qbinom{H}{e}}C_U\quad\text{(disjoint).}
\]

\begin{lem}\label{lem:equi}
With the above notation,
$\cC$ is an equitable partition of the subgraph of 
$J_q(2e+1,e+1)$ induced on $\cA$.
\end{lem}
\begin{proof}
Observe that, the partition
$\{\{U,\sigma(U)\}\mid U\in\qbinom{H}{e}\}$
of the graph $J_q(2e,e)$ with vertex set $\qbinom{H}{e}$, is equitable.
This is because
$\dim U\cap U'=\dim \sigma(U)\cap \sigma(U')$ for any 
$U,U'\in\qbinom{H}{e}$.
The result follows from Lemma~\ref{lem:1}.
\end{proof}

Let $U\in\qbinom{H}{e}$ and $W\in D$. Since 
\[
\{W_1\in C_U\mid \dim W_1\cap W=e\}=\begin{cases}
C_U&\text{if $W\supset U$,}\\
\emptyset&\text{otherwise,}
\end{cases}
\]
we have
\begin{align*}
&|\{W_1\in C_U\cup C_{\sigma(U)}\mid \dim W_1\cap W=e\}|
\\&=\begin{cases}
|C_U\cup C_{\sigma(U)}|&\text{if $W\supset U+\sigma(U)$,}\\
|C_U|&\text{if $W\supset U$ and $W\not\supset\sigma(U)$,}\\
|C_{\sigma(U)}|&\text{if $W\not\supset U$ and $W\supset\sigma(U)$,}\\
0&\text{otherwise}
\end{cases}
\\&\in
\{|C_U\cup C_{\sigma(U)}|,\frac12|C_U\cup C_{\sigma(U)}|,0\}.
\end{align*}
This implies that the partition
\begin{equation}\label{p}
\{C_U\cup C_{\sigma(U)}\mid U\in\qbinom{H}{e}\}\cup\{D\}
\end{equation}
of $\qbinom{V}{e+1}$ satisfies the hypothesis of Godsil--McKay
switching. Let $\tilde{\Gamma}$ denote the resulting graph.
Then for $W_1\in C_U$ and $W_2\in D$,
\begin{equation}\label{ta}
W_1\sim W_2\text{ in }\tilde{\Gamma}\iff
W_2\supset\sigma(U).
\end{equation}

\section{The isomorphism}

In this section, we prove our main result.

\begin{theorem}
The graph $\tilde{\Gamma}$ obtained
by Godsil--McKay switching to the Grassmann graph $J_q(2e+1,e+1)$
with respect to the partition \eqref{p}
is isomorphic to the twisted Grassmann graph $\tilde{J}_q(2e+1,e+1)$.
\end{theorem}
\begin{proof}
Since $\tilde{J}_q(2e+1,e+1)$ is isomorphic to the 
block graph $\Delta(e,q)$ by \cite{MT}, 
it suffices to give an isomorphism between $\tilde{\Gamma}$
and $\Delta(e,q)$.
We claim that $\phi:\qbinom{V}{e+1}\to \cA'\cup\cB'$
defined by
\[
\phi(W)=\begin{cases}
[\sigma(W\cap H)]\cup[W\setminus H]&\text{if $W\in\cA$,}\\
[W]&\text{otherwise}
\end{cases}
\]
is an isomorphism from $\tilde{\Gamma}$ to $\Delta(e,q)$.

Let $W_1,W_2\in\qbinom{V}{e+1}$. First suppose
$W_1,W_2\in\cA$. Then
\begin{align*}
|[W_1\cap W_2]|&=
|[W_1\cap H]\cap [W_2\cap H]|+|[W_1\setminus H]\cap[W_2\setminus H]|
\nexteq
|[\sigma(W_1\cap H)]\cap [\sigma(W_2\cap H)]|
+|[W_1\setminus H]\cap[W_2\setminus H]|
\nexteq
|([\sigma(W_1\cap H)]\cup[W_1\setminus H])\cap
 ([\sigma(W_2\cap H)]\cup[W_2\setminus H])|
\nexteq
|\phi(W_1)\cap\phi(W_2)|.
\end{align*}
We also have $|[W_1\cap W_2]|=|\phi(W_1)\cap\phi(W_2)|$
if $W_1,W_2\in D$. Therefore, for these two cases,
\begin{align*}
W_1\sim W_2\text{ in }\tilde{\Gamma}
&\iff
W_1\sim W_2\text{ in }J_q(2e+1,e+1)
\nexteqv
|[W_1\cap W_2]|=\frac{q^e-1}{q-1}
\nexteqv
|\phi(W_1)\cap\phi(W_2)|
=\frac{q^e-1}{q-1}
\nexteqv
\phi(W_1)\sim\phi(W_2)\text{ in }\Delta(e,q).
\end{align*}

Next suppose 
$W_1\in\cA$, $W_2\in D$. Then there exists $U\in\qbinom{H}{e}$
such that $W_1\in C_U$. 
By \eqref{ta}, we have
\begin{align*}
W_1\sim W_2\text{ in }\tilde{\Gamma}
&\iff
W_2\supset\sigma(U)
\nexteqv
[\sigma(U)]\cap [W_2]=[\sigma(U)]
\nexteqv
|([\sigma(W_1\cap H)]\cup[W_1\setminus H])\cap [W_2]|=|[\sigma(U)]|
\nexteqv
|\phi(W_1)\cap\phi(W_2)|
=\frac{q^e-1}{q-1}
\nexteqv
\phi(W_1)\sim\phi(W_2)\text{ in }\Delta(e,q).
\end{align*}
\end{proof}

Note that the Godsil--McKay switching we have described depends
on a polarity of the hyperplane $H$. One might wonder whether
different choice of a polarity gives rise to nonisomorphic
graphs. This question has already been addressed in the context
of pseudo-geometric designs in \cite{JT}. Since the composition
of two polarities is a collineation of the projective space
defined by $H$, and every collineation of $H$ extends to that of $V$,
the resulting switched graphs are isomorphic. The fact that
the resulting graph is not isomorphic to the original Grassmann graph
is related to the existence of an extra automorphism, i.e., a polarity,
of the Grassmann graph $J_q(2e,e)$ with vertex set $\qbinom{H}{e}$,
which does not extend to an automorphism of $J_q(2e+1,e+1)$.

\subsection*{Acknowledgements}
The author would like to thank Alexander Gavrilyuk for
helpful discussions.

\end{document}